\documentclass[oneside,12pt]{amsart}
\usepackage{amscd,amsmath,latexsym}
\usepackage[mathcal]{euscript}
\usepackage{bm,pdfpages,verbatim}
\newtheorem{thm}{Theorem}[section]
\newtheorem{cor}[thm]{Corollary}

\newtheorem{prop}[thm]{Proposition}
\theoremstyle{definition}
\newtheorem{defin}[thm]{Definition}

\newtheorem{remark}[thm]{Remark}

\theoremstyle{remark}

\newtheorem*{remex}{Example}
\newtheorem*{ack}{Acknowledgments}

\def\nd{\noindent}

\def\C{{\mathbb C}}

\begin{document}

\author[A.~Bahri]{A.~Bahri}
\address{Department of Mathematics,
Rider University, Lawrenceville, NJ 08648, U.S.A.}
\email{bahri@rider.edu}

\author[M.~Bendersky]{M.~Bendersky}
\address{Department of Mathematics
CUNY,  East 695 Park Avenue New York, NY 10065, U.S.A.}
\email{mbenders@xena.hunter.cuny.edu}

\author[F.~R.~Cohen]{F.~R.~Cohen}
\address{Department of Mathematics,
University of Rochester, Rochester, NY 14625, U.S.A.}
\email{cohf@math.rochester.edu}

\author[S.~Gitler]{S.~Gitler}
\address{Department of Mathematics,
Cinvestav, San Pedro Zacatenco, Mexico, D.F. CP 07360 
Apartado
Postal 14-740, Mexico} 

\title[On the free loop space of a toric space]{On the free loop spaces of  a toric space}

\



\

\begin{abstract}
In this note, it is shown that the Hilbert-Poincar\'e series for the 
rational homology of the free loop space on a moment-angle complex is a rational function if and only if the
moment-angle complex is a product of odd spheres and a disk. A  partial  result is included
for the Davis-Januszkiewicz spaces.  The opportunity is taken to correct the result 
\cite[\;Theorem $1.3\hspace{0.6mm}$]{bbcg.rational} which used a theorem from \cite{bj}.
\end{abstract}

\maketitle

{\bf This paper is dedicated to Samuel Gitler Hammer who brought us much joy and interest in Mathematics. }

\section{Introduction}\label{Introduction}

Let $$Z_K = Z(K;(D^2,S^1))$$ be a moment-angle complex, 
(a special case of a polyhedral product), 
where $K$ is a finite simplicial complex with $m$ vertices of dimension $n-1$ \cite{ Buchstaber-Panov, 
Buchstaber-Panov.2,  bbcg.1}. In the special cases for which $K$ is a polytopal sphere, $Z_K $ is a manifold with orbit space given by a simple convex polytope $$P^n(K)  = Z_K/T^m$$ where the torus of rank $m$, $T^m$, acts naturally on $Z_K$.
The topology/geometry of the free loop space of the Davis-Januszkiewicz space
$DJ(K) = ET^m \times_{T^m}Z_K$ and related spaces here is tightly tied to the geometry of $P^n(K)$.

\

F\'elix and Halperin showed, \cite{fh} and \cite{fht}, that there is a dichotomy for simply-connected 
finite $CW$-complexes $X$. Their theorem is the following.
\newpage
\begin{thm}\label{thm:dichotomy}
Either
\begin{enumerate}
\item[(1)] {\bf  $\pi_*(X) \otimes  \mathbb{Q}$} is a finite $\mathbb{Q}$-vector space, in which case $X$ is called
rationally elliptic or,
\item[(2)]  {\bf  $\pi_*(X) \otimes  \mathbb{Q}$} grows exponentially, in which case $X$ is called  rationally hyperbolic.
\end{enumerate}
\end{thm}

\

The purpose of this note is to develop the dichotomy in the next Theorem  \ref{thm:rationality.first.result} arising 
from $LX$ the free loop space of a space $X$ together with the connections to $P^n(K)$. For a definition of the
term {\em exponential growth\/}, see  \cite[page 9]{mp}. 

\begin{thm} \label{thm:rationality.first.result}
The Hilbert-Poincar\'e series for the 
rational homology of $$LZ(K;(D^2,S^1))$$   has exponential growth 
 if and only if  $Z(K;(D^2,S^1))$ contains a wedge of two spheres as a  rational  retract, and so is hyperbolic. Thus the following are equivalent:

\begin{enumerate}
\item The Hilbert-Poincar\'e series for the rational homology of $LZ(K;(D^2,S^1))$  has sub-exponential growth.

\item The space $Z(K;(D^2,S^1))$ has totally finite rational homotopy groups, in other words $Z(K;(D^2,S^1))$ is elliptic.
\end{enumerate}

\end{thm}

\

The previous theorem follows  by combining theorems of Pascal Lambrechts \cite{Pascal},  Neisendorfer and Miller
\cite{nm} together with Theorem $1.3$ of \cite{bbcg.rational}, which illustrates this dichotomy in the case  
of $Z_K.$ (The opportunity is taken here to correct this result  in Section \ref{sec:correction}.) The growth of free loop spaces has also been developed in \cite{Felix etal}. 

\

Gurvich in his thesis  \cite{Gurvich}  showed that in the case $K$ is  a  polytopal sphere, 
then $Z_K$ is elliptic if and only if
$P^n(K)$ is a product of simplices. (This result is generalized for any $K$ in \cite{bbcg.rational}).
The next corollary follows from Gurvich's result together with Theorem  \ref{thm:rationality.first.result}.

\begin{cor} \label{cor:rationality for polytopal macs}
Let $K$ be a polytopal sphere. 
Then following are equivalent:

\begin{enumerate}
\item The Hilbert-Poincar\'e series for the rational homology of $LZ(K;(D^2,S^1))$  has sub-exponential growth. 
\item The space $Z(K;(D^2,S^1))$ is elliptic, and so has totally finite rational homotopy groups.
\item The simple polytope $P^n(K)$ is a product of simplices.
\end{enumerate}

\end{cor}

\

In what follows,  a related theorem is stated in which  $Z(K;(D^2,S^1))$ is replaced by either $DJ(K)$ 
the associated Davis-Januszkiewicz space  or  mildly more general spaces.

\ 

Remarks addressing earlier work on irrational Hilbert-Poincar\'e series follow next. J.~E.~Roos first proved that the Hilbert-Poincar\'e series for the free loop space of $S^3 \vee S^3$ is irrational \cite{Roos}, following Serre's method for proving that the Hilbert-Poincar\'e series  for $\Omega^2(S^3 \vee S^3)$ is irrational \cite{serre}. One common theme here is the application of the Lech-Mahler-Skolem theorem which identifies whether certain infinite series are given by rational functions 
\cite{serre, Roos}. However, it is unclear whether these methods extend
directly to many of the cases in this paper.

\

A result due to Pascal Lambrechts is described next \cite{Pascal}. Lambrechts proves that if $X$ is a coformal, $1$-connected CW complex of finite type, and is hyperbolic, then the rational Betti numbers of the free loop space have exponential growth. Examples are wedges of two spheres each of dimension greater than $1$.
(Aside: Let X be a simply connected CW complex with rational
cohomology of finite type. Let $\Lambda( V; d)$ denote the Sullivan minimal model for $X$. Then $\Lambda( V; d)$  is said to be coformal provided $d^2(V) \subset \Lambda^2V.$ )

\

By Theorem $1.3$ in \cite{bbcg.rational},  (corrected below),  either $Z_K$ is rationally homotopy equivalent to a finite product of odd spheres in which case $Z_K$ is elliptic, or rationally 
 $Z_K$  has a wedge of two spheres 
both of dimension greater than one as a retract in which case, it is hyperbolic. The structure of the minimal non-faces
determines whether the moment-angle complex is elliptic or hyperbolic.

A related result holds for the Davis-Januszkiewicz spaces and mild generalizations.

\begin{thm} \label{thm:2}

Let $X =  DJ(K) \  \mbox {or} \ X = ET^m \times_{T^q}Z_K $ where $T^q \subset T^m$. Then if the 
space $Z(K;(D^2,S^1))$ is elliptic (and so has totally finite rational homotopy groups), the Hilbert-Poincar\'e 
series for the rational homology of $LX$ has sub-exponential growth. 
\end{thm} 

\begin{remex} 
  Let $K$ be the simplicial complex consisting of two disjoint points and $Q$ a simplicial complex with
  one edge and a disjoint point. Then, $Z(K;(D^{2},S^{1})) = S^{3}$ is elliptic, and  
  $Z(Q;(D^{2},S^{1}))$ is a wedge of spheres and so is hyperbolic. Further, 
$$ DJ(K) \simeq \mathbb{CP}^{\infty} \vee \mathbb{CP}^{\infty}.$$

\nd  On the other hand, the Hilbert-Poincar\'e 
series for the rational homology of
$$LDJ(Q) \simeq  L\big(\big(\mathbb{CP}^{\infty} \times  \mathbb{CP}^{\infty}\big) \vee \mathbb{CP}^{\infty}\big)$$
\nd may have exponential growth.
\end{remex} 

Since the Hochschild homology of the cohomology ring for
$DJ(K)$ is the cohomology of the free loop space of $DJ(K)$ 
as a special case of \cite{Goodwillie}, the next result follows.

\begin{cor} \label{cor:Hochschild of face ring}

The Hochschild homology of the Stanley-Reisner ring (or face ring of $K$) has Hilbert-Poincar\'e series 
having 
sub-exponential growth,   if the space $Z(K;(D^2,S^1))$ is elliptic. 
Furthermore, if $K$ is a polytopal sphere, the Hochschild homology of the Stanley-Reisner ring has Hilbert-Poincar\'e series which is a rational function if the simple polytope $P^n(K)$ is a product of simplices. 

\end{cor}

A related question is to work out the precise cohomology of $LX$.
In the paper \cite{fadh}, Fadell and Husseini computed the cohomology ring of $LM$ for $M$ a sphere or a complex projective space.  The  Chas--Sullivan rings of the homology of these $M$ have been computed by Cohen, Jones and Yan in \cite{cjy}.  Using more elementary means, the calculation has been done also by N.~Seeliger \cite{Seeliger}. 
In the special case for which $Z(K;(D^2, S^1))$ is rationally elliptic, the homology of the free
loop space is just that of a product of odd dimensional spheres with a product of pointed loop spaces of odd dimensional spheres.  To work out the homology of $LDJ(K)$ in the rationally elliptic case, it suffices to work out the differentials
in the spectral sequence for $L(Z(K;(D^2, S^1))) \to L(DJ(K)) \to L( \mathbb C \mathbb P(\infty))^m$ where there is a homotopy equivalence $$ L( \mathbb C \mathbb P(\infty))^m \to   \mathbb C \mathbb P(\infty)^m \times (S^1)^m.$$ 
\

The examples above arise from Ganea's fibration $$S^3 \to  \C \mathbb P^{\infty} \vee \C \mathbb P^{\infty} \to \C \mathbb P^{\infty} \times \C \mathbb P^{\infty}.$$ In this case $K$ has two vertices without an edge between the vertices, $Z_K = S^3$, and $DJ(K) = \C \mathbb P^{\infty} \vee \C \mathbb P^{\infty}$. The upshot is that Hilbert-Poincar\'e series for 
$L(\C \mathbb P^{\infty} \vee \C \mathbb P^{\infty})$  has sub-exponential growth. 

\

\section{The dichotomy for $Z(K;(D^{2},S^{1}))$: a correction to \cite[\;Theorem $1.3\hspace{0.6mm}$]{bbcg.rational}}\label{sec:correction}

In the paper \cite{bbcg.rational}, a result from \cite{bj} is used to prove that the moment-angle complex
$Z(K;(D^{2},S^{1}))$ is rationally elliptic if and only if it is the product of odd spheres and a disk. This
occurs if and only if $K$ is the iterated join of simplices and boundaries of simplices.

Recently, counterexamples to the relevant result from \cite{bj} have appeared in the literature. This 
necessitates a repair to \cite[Theorem $1.3$]{bbcg.rational} which is included below. Our goal is to prove
that if a simplicial complex $K$ does not have pairwise disjoint non-faces, then rationally, $Z(K;(D^{2},S^{1}))$
has a wedge of odd spheres as a retract, and so it will be rationally hyperbolic. Notice that, by
\cite[Corollary $2.7$]{bbcg.rational},  the hypothesis here is equivalent to $K$ not being the iterated 
join of simplicies and boundaries of simplices. The next proposition will reduce the proof to a simple induction.

\begin{defin}
Let $\mathcal{A}_{m}$ be the collection of all simplicial complexes on $m$ vertices which have
a pair of intersecting minimal non-faces, but no proper full subcomplex with that property.
\end{defin}

\begin{remex}
Let $m=4$ and $K$ have minimal non-faces corresponding to relations in the Stanley-Reisner ring:
$v_{1}v_{2}v_{3}$, $v_{1}v_{2}v_{4}$ and $v_{1}v_{4}$. Here, $K$ has no proper full subcomplex
with intersecting non-faces.
\end{remex}

\begin{prop}\label{prop:am}
Let $K \in \mathcal{A}_{m}$, then $Z(K;(D^{2},S^{1}))$ has a wedge of odd spheres as a retract.
\end{prop}
\begin{proof}
Suppose that $K$ has minimal intersecting non-faces corresponding to the following relations in the 
Stanley-Reisner ring
$$v_{1}\cdots v_{k}w_{1}\cdots w_{t}\quad\text{and}\quad u_{1}\cdots u_{r}w_{1}\cdots w_{t}.$$

\nd (Notice that minimality dictates that $k$, $t$ and $r$ are all $\geq 1$.) 
It follows that the vertex set of $K$ must be
\begin{equation}\label{eqn:vertices}
\big\{v_{1},\ldots, v_{k},u_{1},\ldots,u_{r},w_{1},\ldots,w_{t}\big\}
\end{equation}

\nd for otherwise, removing a vertex from $K$, which is not among these, will produce a proper full
subcomplex contradicting $K \in \mathcal{A}_{m}$.  Next, setting
$$I =\big\{v_{1},\ldots, v_{k},w_{1},\ldots,w_{t}\big\} \quad\text{and}\quad 
J = \big\{u_{1},\ldots,u_{r},w_{1},\ldots,w_{t}\big\}$$

\nd gives retractions off $Z(K;(D^{2},S^{1}))$:
$$Z_{K_{I}} = S^{2(k+t) -1} \quad\text{and}\quad Z_{K_{J}} = S^{2(r+t) -1}.$$

\nd corresponding to the full subcomplexes $K_{I}$ and $K_{J}$,  \cite[Theorem $2.2.3$]{ds} . The stable splitting theorem of \cite{bbcg.1} distinguishes these two spheres. This gives a map
$$S^{2(k+t) -1} \vee S^{2(r+t) -1} \longrightarrow Z(K;(D^{2},S^{1})).$$

\nd It remains to show that rationally, no cells are attached to this wedge of spheres inside $Z(K;(D^{2},S^{1}))$.
Now, the results of \cite{bbcg.1} imply that all non-trivial attaching maps to this wedge of spheres must be 
{\em stably trivial\/}.
The Hilton-Milnor theorem, \cite[Theorem $4.3.2$]{jn}, gives
\begin{align*}\pi_{n}(S^{2(k+t) -1} \vee S^{2(r+t) -1})&\cong \pi_{n}(S^{2(k+t) -1})\oplus  \pi_{n}(S^{2(r+t) -1})\oplus 
\pi_{n}\big(\Sigma(S^{2(k+t) -2} \wedge S^{2(r+t) -2})\big)\\
&\oplus_{j \geq 2} \pi_{n}\big(\Sigma(S^{2j(k+t) -j} \wedge S^{2(r+t) -1})\big).                       
\end{align*}

\nd The rational homotopy groups of spheres is well known. The only stably trivial non-trivial classes occur in the groups
$\pi_{4q-1}(S^{2q})$.  In the decomposition above, this requires
$$n \geq  4(2k+3t+r-1)-1.$$

\nd The vertex set of $K$ is given by \eqref{eqn:vertices} and so the largest cell possible in $Z(K;(D^{2},S^{1}))$
has dimension $2(k+r+t)-1$. Now
$$2(k+r+t)-1 < 4(2k+3t+r-1)-1$$

\nd because $k$, $t$ and $r$ are all $\geq 1$. So rationally, no non-trivial attaching map is possible. \end{proof}

An induction argument now gives the result.

\begin{thm}\label{thm:induction}
Let $K$ be a simplicial complex which contains a pair of minimal intersecting non-faces, then $Z(K;(D^{2},S^{1}))$
is rationally hyperbolic.
\end{thm}

\begin{proof}
It is straightforward to check that all simplicial complexes on three vertices, which have pairwise intersecting 
non-faces have a wedge of spheres as a retract and so are rationally hyperbolic.

Suppose by way of induction, that all simplicial complexes with fewer than $m$ vertices,
which have pairwise intersecting non-faces, have a wedge of spheres as a rational retract.  Let
$K$ be a simplicial complex on m vertices, which has pairwise intersecting non-faces. If
$K \in \mathcal{A}_{m}$, the result is true for $K$ by Proposition \ref{prop:am}. If $K \notin \mathcal{A}_{m}$,
then $K$ has a proper full subcomplex $L$ which has a pair of intersecting non-faces.  The induction
hypothesis and \cite[Theorem $2.2.3$]{ds} now imply the result. \end{proof}

\section{Proof of Theorem \ref{thm:rationality.first.result} }\label{Proof of Theorem 1.2}

 Assume that $ Z(K;(D^2, S^1))$ is rationally hyperbolic. Thus 
$ Z(K;(D^2, S^1))$ has a rational wedge of two simply-connected spheres as a retract.
A wedge of
two spheres is coformal by a result of Neisendorfer and Miller, \cite[page 573]{nm}. Appealing to Lambrecht's theorem 
\cite{Pascal}, the Hilbert-Poincar\'e series for the rational homology of the free loop space of $ Z(K;(D^2, S^1))$  has exponential growth as the Hilbert-Poincar\'e series for the free loop space of a wedge of two simply-connected spheres has exponential growth. Thus the rational homology of $ LZ(K;(D^2, S^1))$ has exponential growth.

\

Conversely, note that $ Z(K;(D^2, S^1))$ is rationally elliptic if and only if it is rationally homotopy equivalent to a product of odd spheres.  The free
loop space of a product of odd spheres is rationally, (or indeed after inverting $2$), homotopy equivalent to the product of odd spheres with the pointed loop space of the finite product of odd spheres.  {\em This product
has a cohomology algebra which has sub-exponential growth\/}.  

\

These remarks imply Theorem  \ref{thm:rationality.first.result} since any space of the homotopy type of a finite, 
 simply  connected CW-complex is either elliptic, or hyperbolic.

\

 \begin{remark} The calculations of the Chas--Sullivan string topology rings of $H_{\ast}(LM)$ 
for $M = S^{n}$ and 
$\mathbb{CP}^{n} $, mentioned above \cite{cjy} and \cite{Seeliger}, yield a quotient of a finitely generated free
associative algebra by an ideal. In particular, 
using the Chas--Sullivan product one sees that the homology of these free loops are rationally elliptic. 
Now $Z(K;(D^{2},S^{1}))$ is a manifold if $K$ is a triangulation of a sphere. So, the string topology rings of 
$LZ(K;(D^{2},S^{1}))$ are defined for such $K$. It follows from Theorem \ref{thm:rationality.first.result} that 
the Chas--Sullivan string topology of the free loops on a moment angle manifold 
$Z(K;(D^{2},S^{1}))$ 
cannot be a quotient of a  finitely generated free associative 
algebra unless $Z(K;(D^{2},S^{1}))$ is a product of odd spheres.
\end{remark} 

\section{Proof of Theorem \ref{thm:2}}\label{Proof of 2}

Suppose condition (1) holds, namely that the rational cohomology $LZ(K;(D^{2},S^{1})$ has sub-exponential growth.
   In this case, $Z(K;(D^{2},S^{1})$ is rationally elliptic and so, by  the results of \cite{bbcg.rational}, must
   be rationally homotopy equivalent to a product of odd spheres.  \nd Recall next,  (\cite[page 339]{buchstaber.panov.new.book}, for example), that there is a homotopy equivalence
\begin{equation}\label{eqn:splitting}
\Omega\big(DJ(K)\big)  \longrightarrow \Omega\big(Z(K;(D^{2},S^{1})\big)\big) \times T^{m}.
\end{equation}

  \nd This
 implies that the rational cohomology of $\Omega\big(DJ(K)\big)$
   is a tensor product of a polynomial algebra and an exterior algebra. Next, the Serre spectral sequence of   
   the fibration 
\begin{equation}\label{eqn:mainfibration}
\Omega\big(DJ(K)\big) \longrightarrow L\big(DJ(K)\big) \longrightarrow DJ(K)
\end{equation}
  
  \nd has an $E_{2}$ term which is a tensor product of a polynomial algebra, an
  exterior algebra and the Stanley-Reisner ring. So, the rational cohomology $L\big(DJ(K)\big)$ must have 
  sub-exponential growth. This completes the proof of the theorem  for the case of $DJ(K)$.
  The proof of  the theorem for the space  $ET^{m} \times_{T^{q}} Z_{K}$ is entirely analogous. 

 \section{Free loop spaces in the elliptic case}\label{elliptic case}
 
Assume that $Z_K = Z(K;(D^2,S^1))$ is rationally elliptic, then it is a finite product of odd dimensional spheres by
 \cite{bbcg.rational}. The free loop space $L(S^{2n+1})$ is homotopy equivalent to 
 $$S^{2n+1} \times \Omega(S^{2n+1})$$ as long as the prime $2$ has been inverted.
In this case of $L(Z_K)$, the free loop space, is  a product of free loop spaces of odd dimensional spheres.   
 \
 
One remark is that the natural spectral sequence for $$L(DJ(K)) \to L( \mathbb C \mathbb P^{\infty})^m$$ 
frequently supports a non-trivial differential as in the case of the free loops of Ganea's fibration $$L(S^3) \to  
L(\C \mathbb P^{\infty} \vee \C \mathbb P^{\infty}) \to L(\C \mathbb P^{\infty} \times \C \mathbb P^{\infty})$$ for which $K$ is two points, and $DJ(K) = \C \mathbb P^{\infty} \vee \C \mathbb P^{\infty}$. This differential propagates to several related cases. 

\

It is natural to conjecture that if $Z_K$ is rationally elliptic, then the Hilbert-Poincar\'e series for the free loop space
of $L(ET^m \times_{T^q}Z_K)$ is a rational function.
 
 \noindent \begin{ack} The authors are grateful to  Ran Levi and Kathryn Hess for useful
 suggestions and also to Jason McCullough.  The comments of the referee have improved
 the exposition. The first author was supported in part by  grant number 210386 from the Simons Foundation. 
 \end{ack}

\bibliographystyle{amsalpha}

\begin{thebibliography}{99}

\bibitem{bbcg.1}  A.~Bahri, M.~Bendersky, F.~R.~Cohen, and S.~Gitler, {\em The Polyhedral Product Functor: a method of computation for moment-angle complexes, arrangements and related spaces }, Advances in Mathematics, 225 (2010), 1634--1668.

\bibitem{bbcg.rational}  A.~Bahri, M.~Bendersky, F.~R.~Cohen, and S.~Gitler, {\em On the rational type of moment-angle complexes},  Proceedings of the Steklov Institute of
Mathematics, Russian Academy of Sciences, 2014, Vol. 286, 219--223, (DOI) 10.1134/S0081543814060121

\bibitem{berg}  A.~Berglund, {\em Homotopy Invariants of Davis-Januszkiewicz spaces and moment angle complexes}. Available at: http://www.math.ku.dk/alexb/papers.html.

\bibitem{bj}  A.~Berglund and M.~J\"ollenbeck,{\em On the Golod property of Stanly-Reisner rings}, Journal of Algebra \textbf{315}, 249--273 (2007).

\bibitem{Buchstaber-Panov} V.~Buchstaber and T.~Panov, {\em Torus actions, combinatorial topology and homological algebra}, Uspekhi Mat. Nauk 55 (2000), no. 5, 3 (106 ) (Russian). Russian Math. Surveys 55 (2000), no. 5, 825(921) (English translation); arXiv:math.AT/0010073.

\bibitem{Buchstaber-Panov.2} V. Buchstaber and T. Panov, {\em Torus actions and their applications in topology and combinatorics}, AMS University Lecture Series, 24, (2002).

\bibitem{buchstaber.panov.new.book}  V.~Buchstaber and T.~Panov, {\em Toric Topology\/}.
Available at: http://arxiv.org/abs/1210.2368

\bibitem{cjy}R.~L.~Cohen, J.~D.~S. Jones, and J.~Yan, {\em The loop homology algebra of spheres and projective spaces}, Categorical decomposition techniques in algebraic topology, Isle of Skye, 2001, volume 215 of Progr. Math., pages 77--92, Birkhauser, Basel, 2004.

\bibitem{ds}  G.~Denham and A.~Suciu, {\em Moment-angle complexes,
monomial ideals and Massey products}, Pure and Applied Mathematics Quarterly
{\bf 3}, no. 1,  25--60, (2007).

\bibitem{fadh} E.~Fadell, S.~Husseini, {\em Infinite cup length in free loop spaces with an application to a problem of the $N$--body type}, Ann. Inst. H. Poincare Anal. Non Lineaire 9 (3) (1992) 305--319. 

\bibitem{Felix etal} Y.~F\'elix, S.~Halperin, Steve, J.-C.~ Thomas, {\em On the growth of the homology of a free loop space}, (English summary)  Pure Appl. Math. Q. 9 (2013), no. 1, 167--187. 


\bibitem{fh}  Y.~Felix, S.~Halperin, {\em Formal spaces with finite-dimensional rational homotopy\/}.
Transactions of the American Mathematical Society, {\bf 270}, no~. 2, 575--588, (1982)

\bibitem{fht}  Y.~Felix, S.~Halperin, J-C.~Thomas, {\em Rational Homotopy Theory}, Graduate Texts in Mathematics, Springer-Verlag, N.Y, (2002).


\bibitem{Goodwillie}T.~G.~Goodwillie, {\em Cyclic homology, derivations, and the free loopspace}, Topology
24 (1985), no. 2, 187--215.

\bibitem{Gurvich}M.~Gurvich, {\em Some results on the topology of quasitoric manifolds and their equivariant mapping spaces}, Ph.D. thesis, (2008),  Series: UC San Diego Electronic Theses and Dissertations

\bibitem{mp}J.~McCullough, I.~Peeva, {\em Infinite Graded Free Resolutions,} to appear in Commutative Algebra and Noncommutative Algebraic Geometry (Eisenbud, Iyengar, Singh, Stafford, Van den Bergh eds.), Math. Sci. Res. Inst. Publ., Cambridge University Press.

\bibitem{jn}J.~Neisendorfer,  {\em Algebraic Methods in Unstable Homotopy Theory},
Cambridge University Press, Cambridge, 2010.

\bibitem{nm}J.~Neisendorfer, T.~Miller, {\em Formal and coformal spaces\/}, Illinois Journal of Mathematics,
22, no. 2, (1978). 

\bibitem{Pascal}P.~Lambrechts, {\em On the Betti numbers of the free loop space of a coformal space}, Journal
of Pure and Applied Algebra 161 (2001), no. 1-2, 177--192. 

\bibitem{Roos}J.-E.~Roos,  {\em Homology of free loop spaces, cyclic homology and nonrational Poincar\'e-Betti series in commutative algebra},  Algebra Some Current Trends  (Varna, 1986), 173--189, 
Lecture Notes in Math., 1352, Springer, Berlin, 1988.

\bibitem{Seeliger}N.~Seeliger, {\em Loop homology algebra of spheres and complex projective spaces}, Forum Mathematicum, DOI: 10.1515/FORM.2011.161. 

\bibitem{serre}J.-P.~Serre,  {\em Un exemple de s\'erie de Poincar\'e non rationnelle}, (French) Nederl. Akad. Wetensch. Indag. Math. 41 (1979), no. 4, 469--471.

\end{thebibliography}

\end{document}